\documentclass[12pt,twoside,reqno]{amsart}
\usepackage[latin1]{inputenc}
\usepackage[english]{babel}
\usepackage{amsmath, amsthm, amsfonts,amssymb}

\usepackage[left=2.5cm,right=2.5cm,top=2cm,bottom=2cm]{geometry}
\newtheorem{thm}{Theorem}[section]
\newtheorem{cor}[thm]{Corollary}
\newtheorem{lem}[thm]{Lemma}
\newtheorem{prop}[thm]{Proposition}
\newtheorem{rem}[thm]{Remark}

\theoremstyle{plain}

\theoremstyle{definition}

\newtheorem{defn}[thm]{Definition}
\theoremstyle{remark}

\newcommand\blfootnote[1]{%
	\begingroup
	\renewcommand\thefootnote{}\footnote{#1}%
	\addtocounter{footnote}{-1}%
	\endgroup
}

\newcommand{\N}{\mathbb{N}}
\newcommand{\R}{\mathbb{R}}

\newcommand{\X}{\mathbb{X}}

\def\ba{\begin{eqnarray*}}
\def\ea{\end{eqnarray*}}

\def\bee{\begin{equation}}

\def\ene{\end{equation}}

\def\la{\langle}
\def\ra{\rangle}

\def\e{\varepsilon}

\title{CHARACTERIZATION OF GREEDY BASES IN BANACH SPACES}
\author{Pablo M. Bern\'a, \'Oscar Blasco}
\date{}
\begin{document}
		\maketitle

		\begin{abstract} We shall present a new characterization of  greedy bases and 1-greedy bases in terms of  certain functionals defined using  distances  to one dimensional subspaces generated by the basis. We also introduce a new property that unifies the notions of unconditionality and democracy and allows us to recover a better dependence on the constants.\end{abstract}
\blfootnote{\hspace{-0.031\textwidth} 2000 Mathematics Subject Classification. 46B15, 41A65.\newline
		The first author was supported by GVA PROMETEOII/2013/013. The second author was supported by the Spanish Project MTM2014-53009-P.}
		
\begin{section}{INTRODUCTION}
	Let $(\X,\Vert \cdot \Vert)$ be a real Banach space and let $\mathcal{B}=(e_n)_{n=1}^{\infty}$ be a semi-normalized Schauder basis of $\X$ with biorthogonal functionals $(e_n^{*})_{n=1}^{\infty}$, i.e, $0<\inf_n\Vert e_n\Vert \leq \sup_n\Vert e_n\Vert<\infty$ and for each $x\in \X$ there exists a unique expansion $x=\sum_{n=1}^\infty e_n^*(x) e_n$.  As usual  $supp(x)=\{n\in \N: e_n^*(x)\ne 0\}$, $|A|$ stands for the cardinal of $A$, $P_A(x)=\sum_{n\in A} e_n^*(x)e_n$ and $1_A=\sum_{n\in A} e_n$. Throughout the paper, we write $\tilde x= (e_n^*(x))_{n\in \N}\in c_0(\N)$, $\|\tilde x \|_\infty=\sup_n |e_n^*(x)|$ and   $x y=0$  whenever $supp(x)\cap supp(y)=\emptyset$. We use the notation $\X_c$ for  the subspace of $\X$ of elements with finite support, i.e. $x\in \X$ and $|supp(x)|<\infty$ or $\tilde x\in c_{00}(\mathbb N)$. Also for each $m\in \N$, $|A|=m$ and $(\e_n)_{n\in A}\in \{\pm 1\}$,  we  denote by $1_{\e A}=\sum_{n\in A} \e_n e_n$, by  $[1_{\e A}]$ the one-dimensional subspace generated by $1_{\e A}$ and  by $[e_n, n\in A ]$ the $m$-dimensional subspace generated by $\{e_n, n\in A\}$.

Recall that a basis $\mathcal{B}$ in a Banach space $\X$ is called \textit{unconditional} if  any rearrangement of the series $\sum_{n=1}^{\infty}e_{n}^{*}(x)e_n$ converges in norm to $x$ for any $x\in \X$. This turns out to be  equivalent the fact that the projections $P_A$
are uniformly bounded on all sets $A$, i.e. there exists a constant $C>0$ such that
\bee\label{su}
\Vert P_A (x)\Vert \leq C\Vert x\Vert,\;\; x\in\X \hbox{ and }\; A\subset\mathbb{N}.
\ene
In such a case we say that $\mathcal B$ is a $C$-suppression
unconditional basis. The smallest constant that satisfies \eqref{su} is the so-called \textit{suppression constant} and it is denoted by $K_{su}$. Moreover, we have that $$K_{su} = \sup\lbrace \Vert P_A\Vert : A\subseteq\mathbb{N} \mbox{\;is finite}\rbrace = \sup\lbrace \Vert P_A\Vert : A\subseteq\mathbb{N} \mbox{\;is cofinite}\rbrace.$$
In particular, for unconditional bases one has that $x=\sum_{n=1}^\infty e^*_{\pi(n)}(x)e_{\pi(n)}$ where $\pi:\N\to \N$ is chosen so that $|e^*_{\pi(n)}(x)|\ge |e^*_{\pi(n+1)}(x)|$ for all $n\in \N$.

For each $x\in \X$ and $m\in \N$, S.V.Konyagin and V.N.Temlyakov defined in \cite{VT}  a \textit{greedy sum} of $x$ of order $m$ by $$G_m(x) = \sum_{n=1}^me_{\pi(n)}^{*}(x)e_{\pi(n)},$$ where $\pi$ is a \textit{greedy ordering}, that is $\pi: \mathbb{N}\longrightarrow\mathbb{N}$  is a permutaion such that $supp(x)=\lbrace j : e_j^{*}(x) \neq 0\rbrace\subseteq \pi(\mathbb{N})$ and $\vert e^*_{\pi(i)}(x)\vert \geq \vert e^*_{\pi(j)}(x)\vert$ for $i\leq j$.  Any sequence $(G_m(x))_{m=1}^{\infty}$ is called  a \textit{greedy approximation} of $x$. Of course we can have several greedy sums of the same order whenever the sequence $(e_j^*(x))_{j=1}^\infty$ contains several terms with the same absolute value.

	Given $x = \sum_{i=1}^{\infty}e_i^{*}(x)e_i \in \X$, we define the \textit{natural greedy ordering} for $x$ as the map $\rho: \mathbb{N}\longrightarrow\mathbb{N}$ such that $supp(x) \subset \rho(\mathbb{N})$ and so that if $j<k$ then either $\vert e_{\rho(j)}^*(x)\vert > \vert e_{\rho(k)}^*(x)\vert$ or $\vert e_{\rho(j)}^*(x)\vert = \vert e_{\rho(k)}^*(x)\vert$ and $\rho(j)<\rho(k)$. The $m$-\textit{th greedy sum} of $x$ is $$\mathcal{G}_m[\X,\mathcal{B}](x):=\mathcal{G}_m(x) = \sum_{j=1}^m e_{\rho(j)}^*(x)e_{\rho(j)},$$
	and the sequence of maps $\lbrace\mathcal{G}_m\rbrace_{m=1}^\infty$ is known as the \textit{greedy algorithm} associated to $\mathcal{B}$ in $\X$.
With this notation out of the way we have that  \bee \label{qG}
\lim_{m\to \infty}\|x- \mathcal G_m(x)\|=0,\; \ene
for any $x\in \X$ whenever $\mathcal B$ is unconditional.

Konyagin and Temlyakov (see  \cite{VT}) also introduced the term of {\it quasi-greedy basis} for the basis satisfying (\ref{qG}) and later Wojtaszczyk (see \cite{Woj}) proved that condition (\ref{qG}) is actually equivalent to the existence of a universal constant $C>0$ such that
\bee \label{qG1}
\|x- \mathcal G_m(x)\|\le C\|x\|,\; x\in \X,\; m\in \N.
\ene
	Of course this means for a (possibly different) constant that
\bee \label{qG2}
\|\mathcal G_m(x)\|\le C\|x\|,\; x\in \X,\; m\in \N.
\ene
Some authors denote the\textit{ quasi-greedy constant} $C_{qg}$ the best one satisfying (\ref{qG2}) while others use the one satisfying both  (\ref{qG1}) and (\ref{qG2}).
Since $\mathcal G_m(x)=P_\Lambda(x)$ for given $\Lambda$ with $|\Lambda|=m$, one has that any $C$-suppression unconditional basis is also $C$-suppression quasi-greedy basis. Hence $C_{qg}\le K_{su}$.

Recently Albiac and Ansorena (\cite[Theorem 2.1]{AA1}) showed that
 $\mathcal B$ is $1$-suppression unconditional if and only  if
$\sup_{m\in \N}\|\mathcal G_m(x)\|\le \|x\|$ if and only if
$\sup_{m,k\in \N}\{\|\mathcal G_m(x)\|, \|x-\mathcal G_k(x)\|\}\le \|x\|$.

 For each $m\in \N$, the $m$-\textit{term error of approximation} with respect to $\mathcal{B}$ is defined as
	$$\sigma_m(x,\mathcal{B}) = \sigma_m(x) := \inf\lbrace d(x, [e_n, n\in A ])\; :\; A\subset \mathbb{N}, \vert A\vert = m\rbrace.$$
	
Clearly $\sigma_m(x)\le \|x- \mathcal G_m(x)\|$.
	Bases where the greedy algorithm is efficient in the sense that the error we make when approximating $x$ by $\mathcal{G}_m(x)$ is comparable with $\sigma_m(x)$ were first considered in \cite{VT} and called greedy bases. Namely a basis $\mathcal{B}$ is said to be \textit{greedy} if there exists an absolute constant $C\geq 1$ such that
	\begin{eqnarray}\label{greedy}
	\Vert x-\mathcal{G}_m(x)\Vert \leq C\sigma_m(x),\; \forall x\in \X,\; \forall m\in\mathbb{N}.
	\end{eqnarray}
	In this case, we will say that $\mathcal{B}$ is $C$-greedy. The smallest constant $C$ that satisfies \eqref{greedy} is the \textit{greedy constant} and is denoted by $C_g$.
	
In the same paper, a basis $\mathcal{B}$ was said to be \textit{democratic} if there is a constant $D\geq 1$ such that \begin{equation}\label{dem}\left\Vert 1_A\right\Vert \leq D\left\Vert 1_B\right\Vert\end{equation} for all $A,B\subset\mathbb{N}$ finite and the same cardinality. The smallest constant appearing in (\ref{dem}) is called the {\it democracy constant} and $\mathcal B$ is said to be a $D$-democratic basis.\newline

\noindent{\bf Theorem KT} \label{vk} (\cite{VT, Tem})
			
{\it (i) If $\mathcal{B}$ is a $C$-greedy basis then $\mathcal{B}$ is $C$-democratic and $C$-suppression unconditional.

(ii) If $\mathcal{B}$ is $K_{su}$- suppression unconditional  and $D$-democratic then $\mathcal{B}$ is $(K_{su}+K_{su}^3D)$-greedy.
		}\newline

Notice that the dependence on the constants is not good enough since $1$-suppression unconditional and $1$-democratic only gives $2$-greedy. To characterize $1$-greedy bases, Albiac and Wojtaszczyk (see \cite{AW}) introduced the so-called Property (A). For each $|S|<\infty$ and $x=\sum_{n\in S} e_n^*(x)e_n\in \X$, we write $n_{max}(x):=\max\{n:n\in S\}$ and $M(x):=\{n\in S: |e_n^*(x)|= \max_m |e^*_m(x)|\}$. A basis is said  {\it to have Property (A)} whenever
\begin{equation} \label{pA} \|x\|=\|\sum_{n\in M(x)} \theta_n(x) e_n^*(x)e_{\pi(n)} + (x-P_{M(x)}x)\|, \end{equation}
for all $\pi: S\to [1,n_{max}(x)]\cap \N$ injective map such that $\pi(j)=j$ if $j\notin M(x)$ and $\theta_n(x)\in \{\pm 1\}$ with $\theta_n(x)=1$ whenever $\pi(n)=n$ for $n\in M(x)$.\newline

\noindent{\bf Theorem AW}\label{Aw} (\cite[Theorem 3.4]{AW})
		{\it Let $\X$ be a Banach space and $\mathcal{B}$ a Schauder basis.	 Then $\mathcal{B}$ is a $1$-greedy basis if and only if $\mathcal B$ is $1$-suppression unconditional and it has Property (A).
}\newline

It has been recently shown by Albiac and Ansorena (see \cite[Theorem 3.1]{AA2}) that the bases with Property (A) coincide with the almost-greedy bases with $C_{ag}=1$ , that is to say \bee \label{ag}\Vert x-\mathcal{G}_m(x)\Vert \leq \underset{\vert A\vert = m}{\inf}\Vert x-P_A(x)\Vert,\; \forall x\in \X,\; \forall m\in\mathbb{N}.\ene
	
Later on Theorem KT and Theorem AW were generalized in \cite{DKOSS} using the so-called Property (A) with constant $C$  (which has been also called  {\it $C$-symmetric for largest coefficients} in \cite{AA2}) where the equality
 (\ref{pA}) is replaced for an inequality
   \begin{equation} \label{pAC} \|x\|\le C \|\sum_{n\in M(x)} \theta_n(x) e_n^*(x)e_{\pi(n)} + (x-P_{M(x)}x)\|.
    \end{equation}

\noindent{\bf Theorem D} \label{DKOSS} (\cite[Theorem 2]{DKOSS})
			
{\it (i) If $\mathcal{B}$ is a $C$-greedy basis then $\mathcal{B}$ is  $C$-suppression unconditional and it has Property (A) with constant $C$.

(ii) If $\mathcal{B}$ is $K_{su}$- suppression unconditional  and Property (A) with constant $C$ then $\mathcal{B}$ is $K_{su}^2C$-greedy.
}\newline

Let us first reformulate Property (A) in terms useful for our purposes (see \cite{DKOSS}).

\begin{lem} \label{lema} Let $\mathcal B$ be a Schauder basis of $\X$. The basis $\mathcal B$ has the Property (A) with constant $C$ if and only if
$$\|  x +1_{\e A} \|\le C \|x+1_{\e' B}  \| $$ for any $\e, \e' \in \{\pm1\}$, $|A|=|B|$, $A\cap B=\emptyset,$ $x\in \X_c$ with  $supp (x)\cap (A\cup B)=\emptyset$ and $\|\tilde x\|_\infty\le 1$.
\end{lem}
\begin{proof} Assume $\mathcal B$ has Property (A) with constant $C$. For each $\e, \e' \in \{\pm1\}$, $A, B$ and $x$ such that $|A|=|B|$, $A\cap B=\emptyset$ and $\|\tilde x\|_\infty\le 1$ with $supp (x)\cap (A\cup B)=\emptyset$, we write
$y=1_{\e A} + x$. Hence $M(y)=A\cup \{n\in supp(x): |e_n^*(x)|=1\}$. Let $\pi:A\to B$ be a bijection and set $\theta_n(y)= \e'_{\pi(n)}$ for $n\in A$. Hence
$\|y\|\le C\|1_{\e' B} + x\|.$

Conversely given $x\in \X_c$ with $supp(x)= S$ and $\alpha=\max\{|e_n^*(x)|: n\in S\}$  one can consider, for each  $\pi$ and $\theta$ in the conditions above, the set $A=\{j\in M(x): \pi(j)\ne j\}$ and define $\e_n=\frac{e_n^*(x)}{|e_n^*(x)|}$ for each $n\in A$. Now, selecting $B=\pi(A)$ and $\e'_n=\theta_n(x)$ for $n\in B$, we have
\ba\left\Vert x \right\Vert&=& \alpha \left\Vert 1_{\e A}+ \frac{1}{\alpha} (x- P_Ax)\right\Vert\\
&\le& C\alpha \left\Vert 1_{\e'B}+ \frac{1}{\alpha}(x- P_Ax)\right\Vert\\
&=&C\left\Vert\sum_{n\in A} \theta_n(x) e_n^*(x)e_{\pi(n)} + (x-P_Ax)\right\Vert.\ea
\end{proof}
	
We would like to introduce here two properties which encode the notions of unconditionality and democracy or unconditionality and Property (A) at once.
\begin{defn} A Schauder basis $\mathcal B$ is said to have  \textit{Property $(Q)$} with constant $C$ whenever
\begin{equation} \label{qp}\|x+ 1_{ A} \|\le C \|x +y+ 1_{ B}\|\end{equation}
 for any  $|A|=|B|$, $A\cap B=\emptyset$ and $x, y\in \X_c$  such that $x y=0$, $\|\tilde x\|_\infty\le 1$ and $supp (x+y)\cap (A\cup B)=\emptyset.$
 \end{defn}

\begin{rem}\label{n0}
 Clearly  \textit{Property $(Q)$} with constant $C$ on a basis $\mathcal B$   implies that $\mathcal B$ is $C$-democratic  and $C$-suppression unconditional.
 Conversely if  $\mathcal B$ is $D$-democratic  and $C$-suppression unconditional then $\mathcal B$ has Property $(Q)$ with constant $C(1+D)$.
\end{rem}

\begin{defn}
Let $z\in \X_c$. We write $\Gamma_0=\X_c$ and for $z\ne 0$ we define \begin{equation}\label{d2}\Gamma_z=\{y\in \X_c: zy=0, |supp(z)|\le|\{n: |e_n^*(y)|=1\}|\}\end{equation}
	\end{defn}

\begin{defn} A Schauder basis $\mathcal B$ is said to have   \textit{Property $(Q^*)$} with constant $C$ whenever
 \begin{equation} \label{qpstar}\|x+ z \|\le C \|x +y\| \end{equation}
 for any   $x,z,y\in \X_c$ such that $xz=0$, $xy=0$, $\max\{\|\tilde x\|_\infty ,\|\tilde z\|_\infty\}\le 1$ and $y\in \Gamma_z$.
 \end{defn}

\begin{rem}\label{q*}
 It is clear that Property $(Q^*)$ implies that $\mathcal{B}$ is suppression unconditional and satisfies the Property (A) with the same constant.

 Conversely if $\mathcal B$ is $K$-suppression unconditional and it has Property (A) with constant $C$ then $\mathcal B$ has Property $(Q^*)$ with constant $KC$.

 Indeed, let $x,z,y\in \X_c$  such that $xz=0$, $xy=0$, $\max\{\|\tilde x\|_\infty , \|\tilde z\|_\infty\}\le 1$ and  $y\in \Gamma_z$. If $z=0$ we have
 $\|x \|\le K \|x +y\|$ using that the basis is $K$-suppression unconditional. Assume now that $z\neq 0$ with $A=supp(z)$. Select $B\subseteq \{n: |e_n^*(y)|=1\}$ with $|B|=|A|$ and $\e'_n=\frac{e_n^*(y)}{e_n^*(y)}$ for $n\in B$. Therefore
 $$
\|x+ 1_{\e A}\|\le  C\|x+ 1_{\e' B}\|\le CK\|x+ y\|.
 $$

 Notice that $\|\tilde z\|_\infty\le 1$ implies that $z\in co(\{1_{\e A}: |\e_n|=1\})$. Hence $x+ z= \sum_{j=1}^m \lambda_j (x+ 1_{\e^{(j)} A})$ for some $|\e^{(j)}_n|=1$ and $0\le \lambda_j\le 1$ with $\sum_{j=1}^m \lambda_j=1$ and we obtain $\|x+z\|\le CK \|x+y\|$.\qed

\end{rem}

 We shall prove in the paper that the Property $(Q)$ and Property $(Q^*)$ are actually equivalent (see Theorem \ref{t0}).

	 In this paper we also introduce two functionals  depending only on distances to one dimensional subspaces which allow us to   characterize the greedy bases and $1$-greedy bases.

	\begin{defn}	Let $\mathcal B$ be a basis in a Banach space $\X$, $x\in \X$ and  $m\in \N$. We define $$\mathcal{D}_m(x,\mathcal{B}) = \mathcal{D}_m(x) := \inf \lbrace d( x,[1_A] ): \; A\subset\mathbb{N}, \vert A\vert = m\rbrace,$$
and
$$\mathcal{D}^*_m(x,\mathcal{B}) = \mathcal{D}^*_m(x):=\inf \lbrace d( x,[1_{\e A}] ): \; (\e_n)\in \{\pm 1\}, A\subset\mathbb{N}, \vert A\vert = m\rbrace.$$
\end{defn}
In particular $$\mathcal{D}^*_m(x,\mathcal{B}) = \mathcal{D}^*_m(x) := \inf \lbrace \Vert x-\alpha(1_{A_1}-1_{A_2})\Vert : \; \vert A_1\cup A_2\vert = m, A_1\cap A_2=\emptyset, \alpha\in\mathbb{R}\rbrace.$$

Of course $\forall x\in \X$ one has $$\sigma_m(x) \leq \mathcal{D}^*_m(x)\leq \mathcal{D}_m(x)\le \|x\|.$$
 Our aim is to show that greedy bases can be actually defined using the functionals $\mathcal{D}^*_m$ or $\mathcal{D}_m$ instead of $\sigma_m$ and that the use of the Property $(Q^*)$ allows to improve the dependence of the constants.\newline

Our main results establish (see Proposition \ref{p0} and Theorem \ref{t1}) that the conditions
$$\Vert x-\mathcal{G}_m(x)\Vert \leq C\mathcal{D}_m(x),\; \forall x\in \X,\; \forall m\in \mathbb{N},$$
$$\Vert x-\mathcal{G}_m(x)\Vert \leq C\mathcal{D}^*_m(x),\; \forall x\in \X,\; \forall m\in \mathbb{N},$$
implies the Property $(Q)$ and Property $(Q^*)$ with constant $C$ respectively and also that bases having Property $(Q^*)$ with constant $C$ are $C^2$-greedy bases. This improves the constants in Theorem D and recover Theorem AW. Combining the results above one gets the following chain of equivalent formulations of greedy bases.
	\begin{cor}\label{t2}
		Let $\X$ be a Banach space and $\mathcal{B}$ a Schauder basis of $\,\X$. The following statements are equivalent:

(i) $\mathcal{B}$ is greedy.

(ii) There exists an absolute constant $C>0$ such that $$\Vert x-\mathcal{G}_m(x)\Vert \leq C\mathcal{D}^*_m(x),\; \forall x\in \X,\; \forall m\in \mathbb{N}.$$

(iii) There exists an absolute constant $C>0$ such that
    $$\Vert x-\mathcal{G}_m(x)\Vert \leq C\mathcal{D}_m(x),\; \forall x\in \X,\; \forall m\in \mathbb{N}.$$

 (iv) $\mathcal{B}$ satisfies the $Q$-property.

(v) $\mathcal{B}$ satisfies the $Q^*$-property.

 (vi) $\mathcal{B}$ is unconditional and democratic.

\end{cor}

 \begin{cor} Let $\mathcal{B}$  be a Schauder basis of $\X$. Then $\mathcal{B}$ is $1$-greedy if only if $\mathcal{B}$ satisfies the $Q^*$-property with constant $1$ if and only  if $\mathcal B$ is $1$-unconditional and it has Property (A) with constant $C=1$.
  \end{cor}	
Our proofs will follow closely the ideas in \cite{AA2,  AW, DKOSS, VT}.

	\end{section}

	
\begin{section}{Some properties of the new functionals $\mathcal{D}_m$ and $\mathcal{D}^*_m$}
Of course
$\mathcal{D}_1(x)=\mathcal{D}^*_1(x)= \|x-\mathcal G_1(x)\|=\|x- e^*_{\rho(1)}(x)e_{\rho(1)}\|.$
However calculating the functionals $\mathcal{D}_m(\cdot)$ and $\mathcal{D}^*_m(\cdot)$ for $m\ge 2$ is not easy in general. Let us study the situation for $\X=\ell^p$ and concrete elements $x\in \X$.
We shall use the following elementary lemmas.

\begin{lem} \label{basiclema} Let $1<p<\infty$ and $m, N\in \N$ such that $m\ge N$. Define, for $\alpha\in \R$ and $ 1\le k\le N$, $$H(\alpha, k)=|1-\alpha|^pk+|\alpha|^p(m-k)+ (N-k)$$
and, for $\alpha\in \R$, $k_1,k_2\in \N$ and $ 1\le k_1+k_2\le N$,
$$L(\alpha, k_1,k_2)=|1-\alpha|^pk_1+|1+\alpha|^pk_2+|\alpha|^p(m-(k_1+k_2))+ (N-(k_1+k_2)).$$
 Then
$$\min_{\alpha\in \R, 1\le k\le N} H(\alpha,k)=\min_{\alpha\in \R, 1\le k_1+k_2\le N} L(\alpha,k_1,k_2)= N\left(1+ \left(\frac{m-N}{N}\right)^{-1/(p-1)}\right)^{-(p-1)}.$$
\end{lem}
\begin{proof}  Using that $H(\alpha, k)\ge H(|\alpha|,k)$ and $L(\alpha,k_1,k_2)= L(-\alpha,k_2,k_1)$ we can restrict to consider $\alpha\in \R^+$. Also since $(\alpha-1)^p k+ \alpha^p (m-k)$  and $(\alpha-1)^p k_1+ (1+\alpha)^p k_2 + \alpha^{p}k_3$ are increasing for $\alpha\ge 1$, the minima are achieved over $0\le \alpha\le 1$.

Let $0\le\alpha\le 1$ and  $0\le k, k_1,k_2\le N$ and $k_1+k_2 \le N$, we write $H(\alpha,k)=H_\alpha(k)= J_k(\alpha)$ that is
$$H_\alpha(k)=\Big((1-\alpha)^p-\alpha^p-1\Big)k+ N+\alpha^pm,$$
Similarly we write $L(\alpha, k_1,k_2)=L_\alpha(k_1,k_2)$ that is $$L_\alpha(k_1,k_2)=\Big((1-\alpha)^p-\alpha^p-1\Big)k_1+ \Big((1+\alpha)^p-\alpha^p-1\Big)k_2+N+\alpha^pm.$$
Since $(1-\alpha)^p\le \alpha^p+1$ and $(1+\alpha)^p\ge \alpha^p+1$ we obtain that
$$\min \{L_\alpha(k_1,k_2):0\le k_1+k_2\le N\}=\min \{H_\alpha(k): 0\le k\le N\}=(1-\alpha)^pN+\alpha^p(m-N).$$

Now the $\min_{0\le\alpha\le 1}J_N(\alpha)$ is achieved at $\alpha_{min}=(1+ (\frac{m-N}{N})^{\frac{1}{p-1}})^{-1}$ and
$$J_N(\alpha_{min})= N(1+ (\frac{m-N}{N})^{-\frac{1}{p-1}})^{-(p-1)}.$$
\end{proof}

\begin{prop} \label{p1} Let  $\X=\ell^p$ for some  $1<  p<\infty$ and $\mathcal B$ the canonical basis. If $B\subset \N$ and $|B|=N$ then
\bee \label{e1}
\mathcal{D}_m(1_B)=\mathcal{D}^*_m(1_B) =(N-m)^{1/p},\;\; m\leq N,
\ene
\bee \label{e2}
\mathcal{D}_m(1_B)=\mathcal{D}^*_m(1_B) = N^{1/p}\left(1+ \left(\frac{m}{N}-1\right)^{-1/(p-1)}\right)^{-1/p'} ,\;\; m\geq N,
\ene
where $p'=\frac{p}{p-1}$.
\end{prop}
\begin{proof}
  Assume first that $m\leq N$. Let  $\alpha\in \R$, $|\e_n|=1$ and $A\subset \N$ with $|A|=m$. Set $1_{\e A}=1_{A_1}-1_{A_2}$. Observe that
\bee \label{basic}\|1_B- \alpha 1_{\e A}\|^p=|1-\alpha|^p\|1_{A_1\cap B}\|^p+|1+\alpha|^p\|1_{A_2\cap B}\|^p+ |\alpha|^p\|1_{A\setminus B}\|^p+ \|1_{B\setminus A}\|^p.\ene
In particular \bee \label{basic2}\|1_B- \alpha 1_{ A}\|^p=|1-\alpha|^p\|1_{A\cap B}\|^p+ |\alpha|^p\|1_{A\setminus B}\|^p+ \|1_{B\setminus A}\|^p.\ene

Therefore $\|1_B- \alpha1_{\e A}\|\ge \|1_{B\setminus A}\|\ge (N-m)^{1/p}$. This gives $\mathcal{D}^*_m(1_B)\ge (N-m)^{1/p}$.

On the other hand, choosing  $A\subseteq B$ and $\alpha=1$ one concludes that $(N-m)^{1/p}=\|1_B-1_A\|\ge \mathcal{D}_m(1_B).$ Therefore we obtain (\ref{e1}).

Assume now that $m\ge N$. Denoting $k= |A\cap B|=\|1_{A\cap B}\|^p$, $k_1= |A_1\cap B|=\|1_{A_1\cap B}\|^p$ and $k_2= |A_1\cap B|=\|1_{A_2\cap B}\|^p$,  we can apply (\ref{basic}) and (\ref{basic2}) together with Lemma \ref{basiclema} to obtain (\ref{e2}).
\end{proof}

\begin{rem} Similar arguments show that for $\X=\ell^1$ and $\mathcal B$ the canonical basis and $B\subset \N$ with $|B|=N$ one has
$\mathcal{D}_m(1_B)=\left\{
  \begin{array}{ll}
     N-m, & \hbox{$m\le N$;} \\
    m-N, & \hbox{$N\le m\le 2N$;} \\
     N, & \hbox{$m\ge 2N$.}
  \end{array}
\right.$
\end{rem}
			For Hilbert spaces and  for orthonormal bases one can compute the functionals explicitely using the inner product.
			\begin{prop} \label{p1}
				Let $\mathbb{H}$ be a Hilbert space and $\mathcal{B}=(e_n)_n$ be an orthonormal basis of $\mathbb{H}$. Then, for $x\in\mathbb{H}$, $$\mathcal{D}_m(x) = \sqrt{\Vert x\Vert^2 - \dfrac{1}{m}\sup\left\lbrace \la x, 1_{A}\ra^2 : \vert A\vert = m \right\rbrace},$$
				$$\mathcal{D}^*_m(x) = \sqrt{\Vert x\Vert^2 - \dfrac{1}{m}\sup\left\lbrace \la x, 1_{\e A}\ra^2 : \vert A\vert = m , (\e_n)\in\{\pm 1\}\right\rbrace}.$$
			\end{prop}
			
			\begin{proof}
				Let $\alpha \in \mathbb{R}$, $(\e_n)\in\{\pm 1\}$ and $\vert A\vert = m$. Then
				\begin{eqnarray*}
					\Vert x-\alpha 1_{\e A}\Vert^2 &=&\Vert x\Vert^2 -2\langle x,\alpha 1_{\e A}\rangle +  \alpha^2\vert A\vert.
				\end{eqnarray*}
				Therefore the minimum of $\Vert x-\alpha 1_{\e A}\Vert^2$ is achieved at $\alpha_0 = \dfrac{\sum_{k\in A}\e_ke_k^{*}(x)}{n}$ and its value is $\Vert x\Vert ^2 - \dfrac{\left(\langle x,\alpha 1_{\e A}\rangle\right)^2}{n}$.
Taking infimum over the corresponding families we obtain the result.
			\end{proof}

Let us point out that  Proposition \ref{p1} gives
$$\lim_{m\longrightarrow\infty}\mathcal{D}_m(1_B)=\lim_{m\longrightarrow\infty}\mathcal{D}^*_m(1_B)= \|1_B\|$$
for any finite set $B$ for  $\X=\ell^p$ and the canonical basis $\mathcal B$.
In fact this  holds true also for any vector $x$ in Hilbert spaces an any orthonormal basis $\mathcal B$.

	\begin{thm}\label{t1}
		If $\mathbb{H}$ is a Hilbert space and $\mathcal{B}=(e_n)_n$ is an orthonormal basis of $\mathbb{H}$, then $$\lim_{m\longrightarrow\infty}\mathcal{D}_m(x)=\lim_{m\longrightarrow\infty}\mathcal{D}^*_m(x) = \Vert x\Vert,\; \forall x\in \mathbb{H}.$$
	\end{thm}
\begin{proof}
				Since $\mathcal{D}^*_m(x)\le \mathcal{D}_m(x) \le \|x\|$, it suffices to see that $\lim_{m\to\infty} \mathcal{D}^*_m(x)=\|x\|$.
				Assume first that $x\in \X_c$ and $supp(x)= B$ with $N=\vert B\vert$. Since, for each $(\e_n)\in \{\pm 1\}$ and $A$ such that $\vert A\vert = m$, we have
				\begin{eqnarray*}
					\dfrac{1}{\vert A\vert}\la x, 1_{\e A}\ra^2 = \dfrac{1}{\vert A\vert}\left(\sum_{k\in A\cap B}\e_ke_k^{*}(x)\right)^2 \leq \Vert x\Vert^2\dfrac{\vert A\cap B\vert}{\vert A\vert}\leq \dfrac{N\Vert x\Vert^2}{m}.
				\end{eqnarray*}
				From Proposition \ref{p1} we conclude that
				\begin{eqnarray*}
					\Vert x\Vert \sqrt{1-N/m}\leq \mathcal{D}^*_m(x)\leq \Vert x\Vert,
				\end{eqnarray*}
				which gives the result for $x\in \X_c$.
				
				For general $x\in \X$, given $\varepsilon>0$, take first $y\in \X_c$ with  $\Vert x-y\Vert < \varepsilon/2$ and observe that $$\mathcal{D}^*_m(x)\geq \mathcal{D}^*_m(y)-\Vert x-y\Vert,$$
				to conclude that $$\lim_{m}\mathcal{D}^*_m(x)\geq \Vert y\Vert -\varepsilon/2 \geq \Vert x\Vert -\varepsilon.$$
				Taking limit as $\varepsilon$ goes to 0, we obtain the result.\end{proof}
				
			\end{section}
			
		
\begin{section}{Bases with  Property $(Q)$ and $(Q^*)$}

\begin{prop} \label{p0} Let $\X$ be a Banach space and $\mathcal{B}$ a Schauder basis of $\X$. The following statements are equivalent:

(i)There exists  $C>0$ such that
    $$\Vert x-\mathcal{G}_m(x)\Vert \leq C\mathcal{D}_m(x),\; \forall x\in \X,\; \forall m\in \mathbb{N}.$$

  (ii)  $\mathcal B$ has Property $(Q)$.

  (iii) $\mathcal B$ is a greedy basis.
 \end{prop}

\begin{proof} Due to Remark \ref{n0} and Theorem KV only the implication (i) $\Rightarrow$ (ii)  requires a proof.

Assume (i).
We shall see first that the basis is democratic. Let $A,B$ with $\vert A\vert = \vert B\vert = n$ and $m = \vert A\setminus B\vert = \vert B\setminus A\vert$. Define, for each $\varepsilon >0$, $x = (1+\varepsilon)1_{A\setminus B}+1_B$ and observe that $\mathcal{G}_m(x) = (1+\varepsilon)1_{A\setminus B}$. Hence,
				\begin{eqnarray*}
					\Vert 1_B\Vert = \Vert x-\mathcal{G}_m(x)\Vert \leq C\mathcal D_m(x)\leq C\Vert x-1_{B\setminus A}\Vert \leq
					C\Vert 1_A\Vert + C\varepsilon\Vert 1_{A\cap B}\Vert.
				\end{eqnarray*}
				Now take the limit as $\varepsilon\to 0$ to complete the argument.

Let us now prove the unconditionality of $\mathcal B$. Let  $x\in \X_c$ and $supp(x)= B$. Let $A\subseteq B$ and write $m=|B\setminus A|$. Select $\alpha >0$ such that $$\alpha > \underset{j\in A}{\sup}\vert e_j^*(x)\vert+\underset{j\in B\setminus A}{\sup}\vert e_j^*(x)\vert,$$
				and define $$y = x+\alpha1_{B\setminus A} = \sum_{j\in B\setminus A}(\alpha + e_j^*(x))e_j + \sum_{j\in A}e_j^*(x)e_j.$$
				Hence $\mathcal{G}_m(y) = \sum_{j\in B\setminus A}(\alpha + e_j^*(x))e_j$ and $P_{A}(x) = y - \mathcal{G}_m(y)$. Then, $$\Vert P_{A}(x)\Vert = \Vert y-\mathcal{G}_m(y)\Vert \leq C\mathcal D_m(y)\leq C\Vert y-\alpha1_{B\setminus A}\Vert = C\Vert x\Vert.$$
\end{proof}

\begin{prop} \label{equi} Let $\mathcal B$ be a basis of $\X$. The following statements are equivalent:

(i) There exists $C>0$ such that \begin{equation} \label{qp*}\|x+ 1_{ \e A} \|\le C \|x + 1_{\e' B}+ y\| \end{equation}
 \noindent for any  $A, B$ such that $A\cap B=\emptyset$ and $|A|=|B|$, any $(\e_n)_{n\in A}, (\e'_n)_{n\in B}\in \{\pm 1\}$ and any $x, y\in \X_c$ such that  $xy=0$, $\|\tilde x\|_\infty\le 1$ and $(A\cup B)\cap (supp(x+y))=\emptyset$.

 (ii) $\mathcal B$ has Property $(Q^*)$ with constant $C$.

 (iii) There exists $C>0$ such that
\begin{equation}\label{ultima}
\|x\|\le C \|x- P_A(x)+ t y \|
\end{equation}
for any  $x\in \X_c$, $t\geq \|\tilde x\|_\infty$, finite set $A$  and $y\in \Gamma_{P_A(x)}$ with $xy=0$.
\end{prop}
\begin{proof}

 (i) $\Rightarrow$ (ii) Let $x,y,z\in \X_c$ with pairwise disjoint supports with $\max\{\|\tilde x\|_\infty, \|\tilde z\|_\infty\}\le 1$ and $y\in \Gamma_z$.

 For $z=0$ we apply (\ref{qp*}) with $A=B=\emptyset$ to obtain $\|x\|\le C\|x+y\|$.

 For $z\ne 0$, denote $A=supp(z)$ and $B_1=\{n\in supp(y): |e^*_n(y)|=1\}$. Since $|B_1|\ge |A|$ we select $B\subseteq B_1$ with $|B|=|A|$ and write $y= P_{B}(y) +P_{B^c}(y)=  1_{\e' B}+ P_{B^c}(y)$ where $\e'_n=\frac{e^*_n(y)}{|e^*_n(y)|}$ for $n\in B$.  From (\ref{qp*}) we have
 $$\|x+ 1_{ \e A} \|\le C \|x + 1_{\e' B}+ P_{B^c}(y)\|=C\|x+y\|,\quad \forall (\e_n)\in \pm 1.$$
 Notice that $\|\tilde z\|_\infty\le 1$ implies that $z\in co(\{1_{\e A}: |\e_n|=1\})$. Hence $x+ z= \sum_{j=1}^m \lambda_j (x+ 1_{\e^{(j)} A})$ for some $|\e^{(j)}_n|=1$ and $0\le \lambda_j\le 1$ with $\sum_{j=1}^m \lambda_j=1$ and we obtain $\|x+z\|\le C \|x+y\|$.

  (ii) $\Rightarrow$ (iii) Let $x,y\in \X_c$ with $xy=0$, $t\geq \|\tilde x\|_\infty$ and a finite set $A$ with $y\in \Gamma_{P_A(x)}$.

  In the case $A\cap supp(x)=\emptyset$ we have $P_Ax=0$ and  from (\ref{qpstar}) one gets $\|\frac{x}{t}\|\le C\|\frac{x}{t}+ u\|$ for any $u\in \X_c$ with $xu=0$.

   In the case $A\cap supp(x)\neq \emptyset$, let $x_1=\frac{x}{t}-P_A(\frac{x}{t})$, $z_1= P_A(\frac{x}{t})$ and $y_1=y$. Since $\max\{\|\tilde x_1\|_\infty, \|\tilde z_1\|_\infty\}\le 1$  and $y\in \Gamma_{z_1}$ we can apply  (\ref{qpstar}) to obtain $$ \|x\|=t\|x_1 + z_1\|\le C\|x-P_A(x)+ ty\|.$$

  (iii) $\Rightarrow$ (i) Let two finite and disjoint sets $A$ and $B$ with $|A|=|B|$, $(\e_n)_{n\in A}, (\e'_n)_{n\in B}\in \{\pm 1\}$,  $x, y\in \X_c$ such that $\|\tilde x\|_\infty\le 1$ with $xy=0$ and $(A\cup B)\cap (supp(x)\cup supp(y))=\emptyset$.  We apply (\ref{ultima}) for $t=1$, the set $A$  and  $u,v\in \X_c$ given by $u= x + 1_{\e A}$ and $v=1_{\e' B}+y$, since $\|\tilde u\|_\infty\le 1$, $v\in \Gamma_{1_{\e A}}$ and $supp(u)\cap supp(v)=\emptyset$.  Therefore
  $$ \|x + 1_{\e A}\|=\|u\|\le C \|u-P_A(u)+v\|=C \|x+ 1_{\e' B}+y\|.$$
This finishes the proof.
  \end{proof}
\begin{lem}\label{lem0}
			Let $\mathcal{B}$ be a Schauder basis of a Banach space $\mathbb{X}$, $x\in \mathbb X$ and a finite set $A$. Then
 $$ \sup\{\|x+1_{\e A}\|: |\e_n|=1\}= \sup\{\|x+u\|: supp (u)=A, \|\tilde u\|_\infty \le 1\}$$
 and
 $$ \sup_{B\subset A}\|x+1_B\| \le \sup\{\|x+1_{\e A}\|: |\e_n|=1\}\le 3 \sup_{B\subset A}\|x+1_B\|.$$
			\end{lem}
			\begin{proof} Denote $$I_1=\sup_{B\subset A}\|x+1_B\|,$$
$$I_2=\sup\{\|x+1_{\e A}\|: |\e_n|=1\},$$
$$I_3= \sup\{\|x+u\|: supp (u)=A, \|\tilde u\|_\infty \le 1\}.$$

Of course $I_1\le I_2$ since each $B\subseteq A$ can be written as $1_B= \frac{1}{2} \big(1_A +(1_B - 1_{A\setminus B})\big)$.

On the other hand $I_2\le I_3$ follows trivially selecting $u= 1_{\e A}$.
The other  inequality $I_2\ge I_3$ follows using the same argument as in Proposition \ref{equi} since  any $u\in \X$ with $\|\tilde u\|_\infty\le 1$ and $supp(u)=A$ satisfies that $u=\sum_{j\in A}e_j^*(y) e_j \in co(\lbrace 1_{\varepsilon A} : \vert \varepsilon_n\vert = 1)$.

			 For the remaining inequality, denote $A^+ := \lbrace j\in A : \varepsilon_j =  1\rbrace$ and $A^- := \lbrace j\in A : \varepsilon_j =  -1\rbrace$. Since  $1_{\varepsilon A} = 1_{A^+}-1_{A^-}$, with $A^+, A^- \subset A$, we can write $x+1_{\e A}= 2(x+1_{A^+})- (1_{A}+x)$ and therefore
 $\|x+1_{\e A}\|\le 3 I_3$
 and we obtain $I_2\le 3I_3$.
		\end{proof}

\begin{thm} \label{t0} Let $\X$ be a Banach space and $\mathcal{B}$ a Schauder basis of $\X$. $\mathcal{B}$ has Property $(Q)$ if and only if $\mathcal{B}$ has Property $(Q^*)$.
 \end{thm}
 \begin{proof}
Of course Property $(Q^*)$
 implies Property $(Q)$.
 Assume that $\mathcal B$  has the Property $(Q)$. In particular
\begin{equation}\label{unc}
\|P_M(z)\|\le C \|z\|,\, z\in \X_c,\, |M|<\infty.
\end{equation}

Let  $|\e_n|=|\e'_n|=1$, $|A|=|B|$, $A\cap B=\emptyset$ and $x, y\in \X_c$ with $xy=0$, $\|\tilde x\|_\infty\le 1$ and $supp (x+y)\cap (A\cup B)=\emptyset.$
By \eqref{unc} and Property $(Q)$, for each $A'\subset A$
\begin{equation}\label{unc2}
\|x+ 1_{A'} \|\le C\| x+1_A\| \le C^2 \|x +y+ 1_{B}\|,\, A'\subset A.
\end{equation}
Applying Lemma  \ref{lem0}, together with  (\ref{unc}) and (\ref{unc2}), we obtain, for $1_{\e' B}=1_{B^+}-1_{B^-}$,
\ba
\|x+ 1_{\e A} \|&\le& 3\sup_{A'\subset A}\| x+1_{A'}\| \le 3C^2 \|x +y+ 1_{B}\|\\
&\le&  3C^2 (\|x +y+ 1_{B^+}\| + \| 1_{B^-}\|) \\
&\le& 6C^3 \|x +y+ 1_{\e' B}\|. \ea
This shows (\ref{qp*}) and therefore $\mathcal B$ has Property $(Q^*)$ invoking Proposition \ref{equi}.
\end{proof}

Let us mention the following result whose proof is borrowed from \cite{AW}.
\begin{prop} \label{estima} Let $\mathcal{B}$ be a $C$-suppression unconditional basis of $\X$. Let $x\in \X_c$, $A\subseteq supp(x)$ and  $\e_n=\frac{e_n^*(x)}{|e_n^*(x)|}$ for $n\in A$. Then
\begin{equation}\label{est}\|\sum_{n\in B}e_n^*(x)e_n+ t1_{\e A}\|\le C \|x\|\end{equation}
for each $B\subset supp(x)\setminus A$ and $t\le \min\{|e_n^*(x)|:n\in A\}$.
\end{prop}
\begin{proof} Given $B\subset supp(x)\setminus A$ and $t\le \min\{|e_n^*(x)|:n\in A\}$ we define
$$f_{t,B}(s)=\sum_{n\in B}e_n^*(x)e_n+ \sum_{n\in A} \chi_{[0,\frac{t}{|e^*_n(x)|}]}(s)e_n^*(x)e_n\in \X_c, \quad 0\le s\le 1.$$
Note that $f_{t,B}(s)=P_{A_s}x$, and then we have that $\|f_{t,B}(s)\|\le C\|x\|$ and $$\sum_{n\in B}e_n^*(x)e_n+ t1_{\e A}=\int_0^1 f_{t,B}(s)ds.$$ Hence 
using vector-valued Minkowski's inequality (\ref{est}) is achieved.
\end{proof}

 \begin{thm} \label{t1} Let $\X$ be a Banach space and $\mathcal{B}$ a Schauder basis of $\X$.

  (i) If there exists  $C>0$ such that
    $$\Vert x-\mathcal{G}_m(x)\Vert \leq C\mathcal{D^*}_m(x),\; \forall x\in \X,\; \forall m\in \mathbb{N},$$
    then $\mathcal B$ has Property $(Q^*)$ with constant $C$.

   (ii) If $\mathcal B$ has Property $(Q^*)$ with constant $C$ then
    $$\Vert x-\mathcal{G}_m(x)\Vert \leq C^2\sigma_m(x),\; \forall x\in \X,\; \forall m\in \mathbb{N}.$$

 \end{thm}

\begin{proof}(i)
 Due to the equivalences in Proposition \ref{equi} we shall show  (\ref{qp*}). Let us take $\e, \e' \in \{\pm1\}$, $|A|=|B|$, $A\cap B=\emptyset$ and $x, y\in \X_c$ such that $xy=0$, $\|\tilde x\|_\infty\le 1$ and $supp(x+y)\cap (A\cup B)=\emptyset$. Let us write $F=supp (y)$, $\eta_n= \frac{e^*_n(y)}{|e^*_n(y)|}$ for $n\in F$ and
define, for each $\delta>0$, $$z= 1_{\e A} +  x + y+  1_{\eta F}+ (1+\delta)1_{\e'B}.$$

Using that $|e^*_n(y+1_{\eta F})|=|\eta_n+ e^*_n(y)|=|e^*_n(y)|(1+\frac{1}{|e^*_n(y)|})\ge 1$ for each $n\in F$ we have $\mathcal G_m(z)=(1+\delta)1_{\e'B}+ y+  1_{\eta F}$ where $m=|B|+ |F|$. Therefore
\ba
\|1_{\e A} +  x\|&=&\|z- \mathcal G_m(z)\|\\
&\le& C\mathcal{D}^*_m(z)\le C \|z- 1_{\e A}-1_{\eta F}\|\\
&=& C\|x + y+(1+\delta)1_{\e'B}\|\\
&\le& C\|x + y+1_{\e'B}\|+\delta mC.\ea
Now taking limit as $\delta$ goes to $0$ one gets (\ref{qp*}).

(ii) By density and homogeneity, it suffices to prove the result when $x$ is finitely supported with $\Vert \tilde{x}\Vert_\infty \leq 1$. Let $x\in \X_c$, $\Vert \tilde{x}\Vert_\infty \leq 1$, $m\in \N$ and let $b\in [e_n : n\in A]$ with $|A|=m$. Select $B$ with $|B|=m$ and $\mathcal G_m(x)= P_B(x).$

 Set $t= \min \lbrace |e_n^*(x)|: n\in B\setminus A\rbrace$ and set $\e_n= \frac{e_n^*(x)}{|e_n^*(x)|}$ for  $n\in  supp(x)$.

  Since $t\geq \| \widetilde{x-P_B(x)}\|_\infty$ we can apply (\ref{ultima}) for $x-P_B(x)$, the set $A\setminus B$, $y=1_{\e (B\setminus A)}$ to obtain
  $$\|x-\mathcal G_m(x)\|\le C \| x-P_B(x)-P_{A\setminus B}(x)+ t1_{\e (B\setminus A)}\|=C\|P_{(A\cup B)^c}(x-b)+t1_{\e (B\setminus A)}\|.$$
  Finally,  since $t \le |e_n^*(x-b)|$ for  $n\in B\setminus A$, applying Proposition \ref{estima} one gets
  $$ \|x-\mathcal G_m(x)\|\le C^2\|x- b\|.$$
  This gives that $\|x-\mathcal G_m(x)\|\le C^2\sigma_m(x)$ and the proof is complete.
\end{proof}

		\end{section}
	
\noindent{\it Acknowledgment:} The authors are indebted to F. Albiac and J.L. Ansorena for providing the manuscript \cite{AA2} and to G. Garrig\'os for useful conversations during the elaboration of this paper.

	
	\medskip
	
	\textsc{Instituto de Matem\'atica Pura y Aplicada, Universitat Polit\`ecnica de Val\`encia, Valencia, 46022, Spain}.\newline
	\textit{E-mail address}: \texttt{pmbl1991@gmail.com}\newline
	\medskip
	
	\textsc{Departamento de An\'alisis Matem\'atico, Universitat de Val\`encia, Campus de Burjassot, Valencia, 46100, Spain}.\newline
	\textit{E-mail address}: \texttt{oscar.blasco@uv.es}\newline
	
\end{document}